\documentclass[a4paper]{article}
\usepackage{authblk}
\usepackage[latin1]{inputenc}
\usepackage[english]{babel}
\usepackage{amsmath,amsfonts,amssymb,amsthm,mathrsfs,dsfont,cite,paralist,cases,aliascnt}
\usepackage[bookmarksnumbered]{hyperref}

\newlength{\mylength}
\makeatletter
\renewcommand{\maketag@@@}[1]{\hbox{\m@th\normalsize\normalfont#1}}%
\makeatother

\DeclareMathOperator{\trace}{Tr}

\newcommand{\Om}{\Omega}
\newcommand{\eps}{\varepsilon}

\newcommand{\FFF}{{\mathcal F}}

\newcommand{\LLL}{{\mathcal L}}
\newcommand{\MMM}{{\mathcal M}}

\newcommand{\PPP}{{\mathcal P}}
\newcommand{\RRR}{{\mathcal R}}
\newcommand{\UUU}{{\mathcal U}}
\newcommand{\YYY}{{\mathcal Y}}
\newcommand{\ZZZ}{{\mathcal Z}}
\newcommand{\NN}{\mathbb{N}}

\newcommand{\RR}{\mathbb{R}}

\newcommand{\DD}{\mathbb{D}}

\newcommand{\FF}{\mathbb{F}}

\newcommand{\LL}{\mathbb{L}}

\newcommand{\YY}{\mathbb{Y}}
\newcommand{\CCCC}{{\mathscr C}}
\newcommand{\HHHH}{{\mathscr H}}

\newcommand{\MMMM}{{\mathscr M}}
\newcommand{\OOOO}{{\mathscr O}}
\newcommand{\SSSS}{{\mathscr S}}

\numberwithin{equation}{section}
\newtheorem{defn}{Definition}[section]
\newaliascnt{lem}{defn}
\newtheorem{lem}[lem]{Lemma}
\aliascntresetthe{lem}

\newaliascnt{prop}{defn}

\aliascntresetthe{prop}

\newaliascnt{theorem}{defn}
\newtheorem{theorem}[theorem]{Theorem}
\aliascntresetthe{theorem}

\newaliascnt{cor}{defn}

\aliascntresetthe{cor}

\theoremstyle{definition}
\newaliascnt{expl}{defn}

\aliascntresetthe{expl}

\newaliascnt{expls}{defn}
\newtheorem{expls}[expls]{Examples}
\aliascntresetthe{expls}

\newaliascnt{rem}{defn}
\newtheorem{rem}[rem]{Remark}
\aliascntresetthe{rem}

\newaliascnt{rems}{defn}

\aliascntresetthe{rems}

\begin{document}

\title{Fully coupled forward-backward stochastic dynamics and functional differential systems\footnotetext{This work is partially supported by ETH Z\"urich and the Oxford-Man Institute (University of Oxford). The authors would like to thank their respective advisors Freddy Delbaen (ETH), Terry Lyons (Oxford) and Zhongmin Qian (Oxford) for their helpful discussions and comments.}}

\date{}

%
\author[$\ast$]{Matteo Casserini}
\affil[$\ast$]{\small{Department of Mathematics, ETH Z\"urich,
Switzerland, \texttt{matteo.casserini@math.ethz.ch}}\vspace{0.4cm}}
\author[$\dagger$]{Gechun Liang \vspace{0.4cm} \!\!}
\affil[$\dagger$]{\small{Department of Mathematics, King's College
London, U.K., \texttt{gechun.liang@kcl.ac.uk}}\vspace{0.2cm}}
\affil[$\dagger$]{\small{Oxford-Man Institute, University of Oxford,
U.K.}\vspace{0.2cm}}

\maketitle

\begin{abstract}

This article introduces and solves a general class of fully coupled
forward-backward stochastic dynamics by investigating the associated
system of functional differential equations. As a consequence, we
are able to solve many different types of forward-backward
stochastic differential equations (FBSDEs) that do not fit in the
classical setting. In our approach, the equations are running in the
same time direction rather than in a forward and backward way, and
the conflicting nature of the structure of FBSDEs is therefore
avoided.

\vspace{0.6cm} \noindent \textit{Keywords.} Backward stochastic
differential equation, BSDE, fully coupled forward-backward
stochastic differential equation, FBSDE, functional differential
equation, functional differential system.

\vspace{0.6cm} \noindent \textit{Mathematics Subject Classification
(2010).} 60H10, 60H30, 93E03.
\end{abstract}


\newpage

\section{Introduction}
Due to their central role at the intersection between stochastic
analysis, mathematical finance and partial differential equations,
backward stochastic differential equations (BSDEs) and
forward-backward stochastic differential equations (FBSDEs) have
been subject of extensive research during the last two decades.
While linear BSDEs had already been introduced by Bismut \cite{Bis}
in 1973, it was only after the seminal work of Pardoux and Peng
\cite{ParPen} in 1990, who first studied the general non-linear
case, that BSDEs gained considerable attention. Since then, the
importance of the theory of BSDEs increased dramatically, finding
numerous applications in stochastic control theory, PDE theory,
mathematical finance and many other fields. We refer the reader to
the books \cite{ElKMaz,MaYon} and the surveys
\cite{ElKPenQue,ElKHamMat} for an extensive overview of BSDEs and
their applications.

While simple types of decoupled FBSDEs had already been considered
by Pardoux and Peng \cite{ParPen2}, the study of fully coupled
FBSDEs has been initiated by Antonelli \cite{Ant}, who proved the
existence and uniqueness of local solutions. The solvability of
fully coupled FBSDEs was later studied by several authors, and
mainly three types of methods have been proposed so far, each having
its constraints and which do not cover each other. The first is the
method of contraction mapping, introduced in the local case by
Antonelli \cite{Ant}: his work was later developed by Pardoux and
Tang \cite{ParTan} to solve FBSDEs globally under additional
monotonicity conditions. Later on, motivated by the method of
continuation in PDE theory, Hu and Peng \cite{HuPen}, Peng and Wu
\cite{PenWu} and Yong \cite{Yon} solved FBSDEs on arbitrary
intervals by relying on a different type of monotonicity assumptions
on the coefficients. Finally, Ma et al. \cite{MaProYon} introduced
the well-known four-step scheme, which links FBSDEs and quasilinear
PDEs: in this case, the coefficients have to be deterministic and
satisfy strong regularity assumptions. This method was further
developed by Delarue \cite{Del}, who relaxed the regularity
assumption on the coefficients by combining the four-step scheme
with the contraction method. In the last years, Zhang
\cite{Zha1,Zha2} and more recently Ma et al. \cite{MaWuZhaZha}
developed a so-called decoupling scheme to solve FBSDEs with random
coefficients, which unifies most of the existing results at least in
the one-dimensional case. For an extensive account of FBSDEs, we
refer to the book by Ma and Yong \cite{MaYon}.

The purpose of this article is to introduce a general class of fully
coupled forward-backward stochastic dynamics on a general filtered
probability space, which contains classical FBSDEs as a special
case. Inspired by the recent work on Lipschitz BSDEs of Liang et al.
\cite{LiaLyoQia}, we study the solvability of these forward-backward
dynamics by introducing an appropriate \emph{functional differential
system} of the form \small
\[
\begin{cases}
d X_t= \mu(t,X_t, \YYY(X,V)_t, \LLL^1(\MMM(X,V))_t) d t + \sigma(t,X_t,\YYY(X,V)_t, \LLL^2(\MMM(X,V))_t) d W_t, \\
d V_t=f(t,X_t, \YYY(X,V)_t, \LLL^3(\MMM(X,V))_t) d t, \\
X_0=x, \quad V_0=0,
\end{cases}
\]
\normalsize where $\mu$, $\sigma$, $f$ are random functions,
$\LLL^1$, $\LLL^2$, $\LLL^3$ are general abstract operators, and
$\MMM$, $\YYY$ are given, for a random function $\phi$, by
\begin{align*}
\MMM(X,V)_t&=\MMM^\phi(X,V)_t:=E[\phi(X_T) + V_T|\FFF_t], \\
\YYY(X,V)_t&=\YYY^\phi(X,V)_t:=\MMM^\phi(X,V)_t - V_t.
\end{align*}

In particular, our approach does not rely a priori on the existence
of martingale representations, and shows that FBSDEs can be
reformulated as functional differential equations defined in a
forward way: since both equations are running in the same time
direction, this avoids the conflicting nature of the structure of
FBSDEs. More important, our results allow us to consider a more
general class of forward-backward stochastic dynamics which are
beyond the existing framework, and extend the results of Liang et
al. \cite{LiaLyoQia} from backward systems to fully coupled
forward-backward systems.

After introducing an appropriate framework and defining properly the
problem, we study its local solvability and derive our main result:
the existence of a unique local solution to the functional
differential system under Lipschitz and monotonicity conditions on
the coefficients and under specific Lipschitz assumptions on the
operators $\LLL^i$. In particular, the conditions on $\LLL^i$ are
rather mild and allow to consider many types of operators different
from the usual martingale integrand processes treated in classical
FBSDEs: as a consequence, we can solve within our framework many
different types of forward-backward equations that do not fit in the
classical FBSDE setting. To emphasize the generality of these
assumptions, we present several examples of possible operators and
potential financial applications.

In the second part of the article, we discuss the solvability of the
system on arbitrarily large time intervals. This is however more
problematic: indeed, it appears impossible to study such an
extension without defining the operators $\LLL^i$ explicitly, and
one has to consider the problem separately for each choice of
$\LLL^i$. We conclude the article by presenting a study of the case
where the filtration is Brownian and the operators $\LLL^i$ are
given by It\^o's representation.

The article is organized as follows. First of all, to provide some
intuition, we give in Section \ref{section_brownian_example} a brief
overview of the functional differential approach in a simple
Brownian setting. In Section \ref{section_problem_setup}, after
introducing a more general framework, we give a rigorous definition
of our class of forward-backward dynamics and the associated
functional differential system. Section \ref{section_local_sol} is
then dedicated to the existence and uniqueness of solutions to the
latter system for sufficiently small time horizons. Finally, in
Section \ref{section_global_sol} we discuss the general problem of
extending the solution to arbitrarily large time intervals, and
study in particular the case of classical Brownian FBSDEs.

\section{The functional differential approach}\label{section_brownian_example}

We would like to begin by providing the reader with some intuition
of the approach we are going to use in the sequel, which is inspired
by the work of Liang et al. \cite{LiaLyoQia}. We first present the
following elementary, but very illustrative result derived in
\cite{LiaLyoQia}:
\begin{rem}[Liang et al. \cite{LiaLyoQia}] \label{rem_intro}
For $T > 0$, assume that we have a special semimartingale $(Y_t)_{t
\in [0,T]}$ on a probability space $(\Om, \FFF, (\FFF_t)_{t \in
[0,T]}, P)$ satisfying the usual assumptions, and let the terminal
value $Y_T= \xi \in L^1(\FFF_T)$ be given. Furthermore, assume that
the canonical decomposition of $Y$ is given by
\[
Y_t=M_t - V_t,
\]
where $M$ is a martingale and $V$ a predictable process of finite
variation with $V_0=0$. Then, if $V_T$ is integrable, it is easy to
verify that, for all $t \in [0,T]$,
\begin{align*}
M_t&=E[M_T|\FFF_t]=E[\xi + V_T|\FFF_t], \\
Y_t&=M_t - V_t=E[\xi + V_T|\FFF_t] - V_t.
\end{align*}
In other words, the semimartingale $Y$ and the martingale $M$ can be
expressed as operators of the terminal value $\xi$ and the finite
variation process $V$.
\end{rem}

We show now, with the help of some intuitive arguments, how this
remark can lead us to an alternative formulation of the classical
FBSDE problem. Let $(\Om, \FFF, P)$ be for the moment a complete
probability space with an $m$-dimensional Brownian motion
$W=(W_t)_{t \in [0,T]}$ and the corresponding filtration
$(\FFF_t)_{t \in [0,T]}$, augmented by the $P$-null sets in $\FFF$.
We consider a classical fully coupled FBSDE of the form
\begin{equation}\label{classfbsde}
\begin{cases}
d X_t = \mu(t,X_t, Y_t, Z_t) d t + \sigma(t,X_t,Y_t,Z_t) d W_t, \\
d Y_t=-f(t,X_t, Y_t, Z_t) d t + Z_t d W_t, \\
X_0=x,\quad Y_T = \phi(X_T),
\end{cases}
\end{equation}
where the functions $\mu: \Omega \times [0,T] \times \RR^n\times
\RR^d \times \RR^{d \times m}\rightarrow \RR^n$, $\sigma: \Omega
\times[0,T]\times \RR^n \times \RR^d \times \RR^{d \times m}
\rightarrow \RR^{n\times m}$, $f: \Omega \times [0,T]\times
\RR^n\times \RR^d \times \RR^{d \times m}\rightarrow \RR^d$, $\phi:
\Omega \times \RR^n\rightarrow \RR^d$ satisfy the usual
measurability and integrability conditions.

Assume now that the above FBSDE has a solution $(X,Y,Z)$. Since
$\Phi(X_T)$ is the terminal value of the semimartingale $Y$,
\autoref{rem_intro} induces us to introduce the operators
$\MMM^\phi$ and $\YYY^\phi$ by defining
\begin{align*}
\MMM^\phi(X,V)_t&:=E[\phi(X_T) + V_T|\FFF_t], \\
\YYY^\phi(X,V)_t&:=\MMM^\phi(X,V)_t - V_t, \quad t \in [0,T],
\end{align*}
for any processes $X$, $V$ such that $\phi(X_T) \in L^1(\FFF_T)$,
$V_T \in L^1(\FFF_T)$. Then, by \autoref{rem_intro} and the
definition of the operators $\YYY^\phi$ and $\MMM^\phi$, it seems
plausible to associate the above FBSDE to the following system of
\emph{forward functional differential equations}
\[
\begin{cases}
d X_t=\mu\big(t,X_t, \YYY^\phi(X,V)_t, \ZZZ^\phi(X,V)_t\big) d t + \sigma(t,X_t,\YYY^\phi(X,V)_t, \ZZZ^\phi(X,V)_t) d W_t,\\
d V_t=f\big(t,X_t, \YYY^\phi(X,V)_t, \ZZZ^\phi(X,V)_t\big) d t, \\
X_0=x, \quad V_0=0,
\end{cases}
\]
where $\ZZZ^\phi$ is given implicitly via It\^o's representation
theorem by
\[
\MMM^\phi(X,V)_T=E[\MMM^\phi(X,V)_T]+\int_0^T \ZZZ^\phi(X,V)_s d
W_s.
\]

The peculiarity of these stochastic differential equations consists
in the fact that the coefficients $\mu$, $\sigma$, $f$ depend not
only on the behaviour of the solution process $(X,V)$ up to the
present value, but also on the terminal value $(X_T, V_T)$ of the
solution: such stochastic differential equations are not standard,
and for this reason they are called \emph{functional} differential
equations (note that the term ``functional differential equations''
is often used in the literature to refer to stochastic delay
differential equations: however, contrary to the latter, the drivers
of the equations studied here do not have any delay in the past, but
rather in the future).

\section{Fully coupled functional differential systems}\label{section_problem_setup}

In this section, we will show how the approach presented above can
be made rigorous and extended to a much more general framework. To
this end, we first need to introduce some notation: in the
following, we fix $T > 0$, and assume that we are given a complete
probability space $(\Om, \FFF, P)$ together with a general
filtration $\FF=(\FFF_t)_{t \in [0,T]}$ satisfying the usual
assumptions. We recall that every $\FF$-martingale has under these
conditions a c\`adl\`ag version, which we will always choose.
Moreover, we denote by $\PPP$ the predictable $\sigma$-field with
respect to $\FF$ and assume that an $m$-dimensional Brownian motion
$W=(W_t)_{t \in [0,T]}$ is defined on $(\Om, \FFF, \FF, P)$.

For $k$, $l \in \NN$, $|\cdot|$ denotes the Euclidean norm on
$\RR^k$, respectively the Hilbert-Schmidt norm on $\RR^{k \times
l}$, and $\RR^{k \times l}$ will often be identified with $\RR^{k
\cdot l}$. We define $\SSSS^2([0,T],\RR^d)$ as the space of all
processes $V:\Om \times [0,T] \to \nolinebreak\RR^d$ continuous and
adapted such that $V_0=0$ and $E[ \sup_{t \in [0,T]} |V_t|^2] <
\infty$, while $\MMMM^2([0,T],\RR^d)$ denotes the space of all
square integrable $\RR^d$-valued martingales on $[0,T]$. Both
$\SSSS^2([0,T],\RR^d)$ and $\MMMM^2([0,T],\RR^d)$ are endowed with
the norm
\[
\|V\|_{\SSSS^2[0,T]}:=\sqrt{E\Big[\sup_{t \in [0,T]} |V_t|^2\Big]},
\]
and note that
$\big(\SSSS^2([0,T],\RR^d),\|\cdot\|_{\SSSS^2[0,T]}\big)$ is then a
Banach space. Sometimes, we will also need the direct sum space
$\SSSS^2([0,T],\RR^d)\oplus \MMMM^2([0,T],\RR^d)$, endowed with the
same norm $\|\cdot\|_{\SSSS^2[0,T]}$.

In the following, we will consider a particular generalization of
the FBSDE \eqref{classfbsde} on the filtered probability space
$(\Om, \FFF, \FF, P)$. More exactly, we will assume that $\mu$,
$\sigma$ and $f$ depend on some general processes $\LLL^i(M)$,
$i=1,2,3$, instead than on $Z$, where $\LLL^i$ is for each $i$ an
abstract operator defined on $\MMMM^2([0,T],\RR^d)$ and taking
values, for some $p_i \in \NN$, in the space of $p_i$-dimensional
adapted processes (the codomains of $\LLL^i$ will be further
specified later). As we will see later more in detail, this
substitution will allow both to take into account the generality of
the filtration $\FF$ and to treat locally other types of
forward-backward equations not fitting in the classical framework.
This generalization of the FBSDE \eqref{classfbsde} leads us to the
following system:
\begin{equation}\label{fbsdyn}
\begin{cases}
d X_t= \mu(t,X_t, Y_t, \LLL^1(M)_t) d t + \sigma(t,X_t,Y_t, \LLL^2(M)_t) d W_t, \\
d Y_t=-f(t,X_t, Y_t, \LLL^3(M)_t) d t + d M_t,  \\
X_0=x, \quad Y_T=\phi(X_T),
\end{cases}
\end{equation}
where $\mu: \Omega \times [0,T] \times \RR^n\times \RR^d \times
\RR^{p_1}\rightarrow \RR^n$, $\sigma: \Omega \times[0,T]\times \RR^n
\times \RR^d\times \RR^{p_2}\rightarrow \RR^{n\times m}$, $f: \Omega
\times [0,T]\times \RR^n\times \RR^d \times \RR^{p_3}\rightarrow
\RR^d$, $\phi: \Omega \times \RR^n\rightarrow \RR^d$ have the
necessary measurability and integrability properties. Note that the
system is then completely determined by $(\mu,\sigma, f, \phi,
\LLL^1, \LLL^2,\LLL^3)$.

\begin{defn}
A solution to \eqref{fbsdyn} is a triplet of processes $(X,Y,M)$
such that $X \in \SSSS^2([0,T],\RR^n)$, $Y \in \SSSS^2([0,T],\RR^d)
\oplus \MMMM^2([0,T],\RR^d)$, $M \in \MMMM^2([0,T],\RR^d)$, and
satisfying the integral formulation of \eqref{fbsdyn}.
\end{defn}

We will call such systems \emph{fully coupled forward-backward
stochastic dynamics}. As mentioned in the previous section, a viable
approach to study the solvability of this system is to reformulate
the problem with the help of appropriate functional differential
equations. To this end, we denote by $\CCCC^\phi_X$ the class of
$\RR^n$-valued adapted processes $X$ on $[0,T]$ such that $\phi(X_T)
\in L^1(\FFF_T)$, by $\CCCC_V$ the class of $\RR^d$-valued adapted
processes $V$ on $[0,T]$ such that $V_T \in L^1(\FFF_T)$, and we
define the operators $\MMM^\phi$ and $\YYY^\phi$ on $\CCCC^\phi_X
\times \CCCC_V$ by
\begin{equation}\label{coupled_y&m}
\begin{split}
\MMM^\phi(X,V)_t&:=E[\phi(X_T) + V_T|\FFF_t], \\
\YYY^\phi(X,V)_t&:=\MMM^\phi(X,V)_t - V_t, \quad t \in [0,T].
\end{split}
\end{equation}

For the rest of this article, we will drop the dependence of
$\MMM^\phi$ and $\YYY^\phi$ on $\phi$ by writing $\MMM$ and $\YYY$.
With the help of the operators $\MMM$ and $\YYY$, we reformulate the
problem \eqref{fbsdyn} as the following fully coupled system of
\emph{forward functional differential equations}: \small
\begin{equation}\label{fbfunctdiffeq}
\hspace{-1em}\begin{cases}
d X_t= \mu(t,X_t, \YYY(X,V)_t, \LLL^1(\MMM(X,V))_t) d t + \sigma(t,X_t,\YYY(X,V)_t, \LLL^2(\MMM(X,V))_t) d W_t, \hspace{-3.6em} \\
d V_t=f(t,X_t, \YYY(X,V)_t, \LLL^3(\MMM(X,V))_t) d t, \\
X_0=x, \quad V_0=0.
\end{cases}
\end{equation}
\normalsize
\begin{defn}
A solution to \eqref{fbfunctdiffeq} is a pair of processes $(X,V)$
such that $(X,V)\in \SSSS^2([0,T],\RR^n) \times
\SSSS^2([0,T],\RR^d)$ and satisfying the integral formulation of
\eqref{fbfunctdiffeq}.
\end{defn}

Such systems will be called \emph{fully coupled functional
differential systems}. It is not difficult to show the equivalence
of the systems \eqref{fbsdyn} and \eqref{fbfunctdiffeq}:

\begin{lem}\label{eq_solvability}
The fully coupled forward-backward system \eqref{fbsdyn} has a
solution if and only if the functional differential system
\eqref{fbfunctdiffeq} does.
\end{lem}
\begin{proof}
If $(X,Y,M)$ solves \eqref{fbsdyn}, then we obtain a solution of
\eqref{fbfunctdiffeq} via the canonical decomposition of the
semimartingale $Y$. Conversely, if $(X,V)$ is a solution of the
functional differential system, then $\big(X, \YYY(X,V),
\MMM(X,V)\big)$ solves \eqref{fbsdyn}.
\end{proof}

An immediate observation is that the functional differential system
\eqref{fbfunctdiffeq} has a more homogeneous structure than that of
the original problem. Indeed, while the forward-backward dynamics
\eqref{fbsdyn} consist of a forward and a backward equation of
different nature, both the functional differential equations in
\eqref{fbfunctdiffeq} are running forward in time and show a similar
dependence of the coefficients on both the present and the terminal
values of the solution processes. In particular, this homogeneity
allows to rewrite the problem more compactly as
\[
\begin{cases}
d \UUU_t= \Psi(t,\pi_1(\UUU_t), \YYY(\UUU)_t, \LLL^1(\MMM(\UUU))_t, \LLL^3(\MMM(\UUU))_t) d t \\
\hfill+ \Sigma(t,\pi_1(\UUU_t), \YYY(\UUU)_t,\LLL^2(\MMM(\UUU))_t) d W_t, \\
\UUU_0=(x,0)^{\mathrm T},
\end{cases}
\]
where $\UUU=(X,V)^{\mathrm T}$, $\pi_1: \RR^{n+d} \to \RR^n$ is the
projection on the first $n$ components, $\Sigma=(\sigma, 0)^{\mathrm
T}$, and for $t \in [0,T]$, $(x,y,z_1,z_3) \in \RR^n\times \RR^d
\times \RR^{p_1} \times \RR^{p_3}$,
$\Psi(t,x,y,z_1,z_3)=(\mu(t,x,y,z_1), f(t,x,y,z_3))^{\mathrm T}$.
For the rest of the article, we prefer however to consider the
system as formulated in \eqref{fbfunctdiffeq}, since it will be more
convenient to treat the coupling between $X$ and $V$ in such a
framework.

\section{Existence and uniqueness of local solutions}\label{section_local_sol}

Because of \autoref{eq_solvability}, we will now focus our attention
on the existence and uniqueness of solutions to the fully coupled
functional differential system \eqref{fbfunctdiffeq}. As a first
step, we study the local solvability of the problem and introduce
sufficient monotonicity and Lipschitz assumptions on the
coefficients $\mu$, $\sigma$, $f$, $\phi$ and the operators
$\LLL^1$, $\LLL^2$ and $\LLL^3$.

In the following, we denote by $\HHHH^2([0,T],\RR^l)$ the space of
$\PPP$-measurable processes $H:\Om \times [0,T] \to \RR^l$ such that
$\|H\|^2_{\HHHH^2[0,T]}:=E[\int_0^T |H_t|^2 d t] < \infty$, where
$l\in \NN$. First of all, we shall assume that the coefficients
$\mu$, $\sigma$, $f$, $\phi$ satisfy the following assumption:

\bigskip
\noindent{\bfseries Assumption (A1):}  \emph{The functions $\mu:
\Omega \times[0,T]\times\RR^n\times \RR^d \times
\RR^{p_1}\rightarrow \RR^n$, $\sigma: \Omega \times[0,T]\times \RR^n
\times \RR^d\times \RR^{p_2}\rightarrow \RR^{n\times m}$, $f: \Omega
\times[0,T]\times \RR^n\times \RR^d \times \RR^{p_3} \rightarrow
\RR^d$ and $\phi: \Omega \times\RR^n\rightarrow \RR^d$ satisfy
Assumption \textup{{\bfseries (A1)}} if there exists a constant
$C>0$ such that:
\begin{itemize}
\item[\textup{(A1.1)}] \parbox[t][][l]{\mylength}{For any $(x,y,z_1,z_2,z_3) \in \RR^n \times \RR^d \times \RR^{p_1} \times \RR^{p_2}\times \RR^{p_3}$, the processes $\mu(\cdot, x,y,z_1)$, $\sigma(\cdot, x,y,z_2)$ and $f(\cdot, x,y,z_3)$ are $\PPP$-measurable and $\phi(x)$ is $\FFF_T$-measurable.
}
\item[\textup{(A1.2)}] \parbox[t][][l]{\mylength}{For every $(x,y,z_1),(x',y',z'_1) \in \RR^n \times \RR^d \times \RR^{p_1}$,
\begin{align*}
(x-x')^{\mathrm T}\big(\mu(\cdot,x,y,z_1) - \mu(\cdot&,x',y,z_1)\big)\le C|x-x'|^2, \\
|\mu(\cdot,x,y,z_1) - \mu(\cdot,x,y',z'_1)| &\le C(|y-y'|+|z_1-z'_1|), \\
|\mu(\cdot,x,0,0)| &\le C(1+|x|) \quad d P \otimes d t \text{-a.s.},
\end{align*}
and the function $x \mapsto \mu(\cdot,x,y,z_1)$ is $d P \otimes d
t$-a.s. continuous. }
\item[\textup{(A1.3)}] \parbox[t][][l]{\mylength}{ $f(\cdot,0,0,0) \in \HHHH^2([0,T],\RR^d)$, $\sigma(\cdot,0,0,0) \in \HHHH^2([0,T],\RR^{n \times m})$ and $\phi(0) \in L^2(\Omega,\RR^d)$.
}
\item[\textup{(A1.4)}] \parbox[t][][l]{\mylength}{For every $(x,y,z_2,z_3),(x',y',z'_2,z'_3) \in \RR^n \times \RR^d \times \RR^{p_2}\times \RR^{p_3}$,
\begin{align*}
|\sigma(\cdot,x,y,z_2)-\sigma(\cdot,x',y',z'_2)|^2
&\le C(|x-x'|^2+|y-y'|^2+|z_2-z'_2|^2), \\
|f(\cdot,x,y,z_3) - f(\cdot,x',y',z'_3)| &\le C(|x-x'|+|y-y'|+|z_3-z'_3|), \\
|\phi(x) - \phi(x')| &\le C |x-x'| \quad d P \otimes d t
\text{-a.s..}
\end{align*}
}
\end{itemize}
}

\begin{rem}
The above conditions on $\mu$, $\sigma$, $f$ and $\phi$ are quite
standard in the theory of FBSDEs (see for instance \cite{MaYon}).
Moreover, the reader can easily verify that the condition (A1.2)
could be replaced by the following stronger, but more standard
assumption:

\medskip
\noindent \textup{(A1.2')}
\parbox[t][][l]{\mylength}{\emph{$\mu(\cdot,0,0,0)  \!\in \!
\HHHH^2([0,T],\RR^n)$ and, for $(x,y,z_1), (x'\!,y'\!,z'_1) \! \in
\! \RR^n \times \RR^d \times \RR^{p_1}\!$,
\[
|\mu(t,x,y,z_1) - \mu(t,x'\!,y'\!,z'_1)| \! \le \!
C(|x-x'|+|y-y'|+|z_1-z'_1|) \;\; d P \otimes d t \text{-a.s.}.
\]
}}
\end{rem}

In particular, the assumptions on $\phi$ allow us to derive the
following Lipschitz estimates on the operators $\YYY$ and $\MMM$,
which will be essential in the sequel:
\begin{lem}\label{fbest_y&m}
Assume that $\phi$ satisfies the conditions in {\bfseries
\textup{(A1)}}. Then we have that, for the operators introduced in
\eqref{coupled_y&m},
\begin{align*}
\YYY&: \SSSS^2([0,T],\RR^n) \times \SSSS^2([0,T],\RR^d) \to \SSSS^2([0,T],\RR^d)\oplus \MMMM^2([0,T],\RR^d), \\
\MMM&: \SSSS^2([0,T],\RR^n) \times \SSSS^2([0,T],\RR^d) \to
\MMMM^2([0,T],\RR^d).
\end{align*}
Moreover, for any $X,X' \in \SSSS^2([0,T],\RR^n)$ and $V,V' \in
\SSSS^2([0,T],\RR^d)$,
\begin{align*}
\|\YYY(X,V) -\YYY(X',V')\|_{\SSSS^2[0,T]} &\le 2 C \|X-X'\|_{\SSSS^2[0,T]}+ 3 \|V-V'\|_{\SSSS^2[0,T]}, \\
\|\MMM(X,V)- \MMM(X',V')\|_{\SSSS^2[0,T]} &\le 2 C
\|X-X'\|_{\SSSS^2[0,T]}+ 2 \|V-V'\|_{\SSSS^2[0,T]}.
\end{align*}
\end{lem}

\begin{proof}
We first prove the second assertion. By the triangle inequality,
\begin{align*}
&\|\MMM(X, V) -  \MMM(X' , V')\|_{\SSSS^2[0,T]} \\
&\le  \big\|E[\phi(X_T)  -  \phi (X'_T)|\FFF_\cdot]\big\|_{\SSSS^2[0,T]}  +  \big\|E[V_T  -  V'_T|\FFF_\cdot]\big\|_{\SSSS^2[0,T]}\\
&= E\Big[ \sup_{t \in [0,T]} \big|E[\phi(X_T)
-\phi(X'_T)|\FFF_t]\big|^2 \Big]^{1/2} + E\Big[\sup_{t \in [0,T]}
\big|E[V_T - V'_T|\FFF_t]\big|^2 \Big]^{1/2}
\end{align*}
and therefore, by Doob's inequality and the assumption on $\phi$,
\begin{align*}
\|\MMM(X, V) \!-\!  \MMM(X'\!, V')\|_{\SSSS^2[0,T]}
\!&\le \!2 \Big(\! E\big[|\phi(X_T) \!-\!\phi(X'_T)|^2\big]^{1/2} \!\!+\! E\big[|V_T \!-\! V'_T|^2\big]^{1/2} \Big) \\
&\le \!2 C \|X-X'\|_{\SSSS^2[0,T]}+2 \|V-V'\|_{\SSSS^2[0,T]}.
\end{align*}
The estimate for $\|\YYY(X,V) - \YYY(X',V')\|_{\SSSS^2[0,T]}$ then
follows by the application of the triangle inequality.
\end{proof}

The next step consists in introducing appropriate conditions for the
abstract operators $\LLL^1$, $\LLL^2$ and $\LLL^3$. These are given
by the following Lipschitz and boundedness assumptions, which are
the same as introduced by Liang et al. \cite{LiaLyoQia}.

\bigskip
\noindent{\bfseries Assumption (L1):} \emph{The operator $\LLL$
satisfies Assumption \textup{{\bfseries (L1)}} if:
\begin{itemize}
\item[\textup{(L1.1)}] \parbox[t][][l]{\mylength}{$\LLL$ maps $\MMMM^2([0,T],\RR^d)$ into $\OOOO^2([0,T],\RR^p)$, where $\OOOO^2([0,T],\RR^p)$ is either the space $\HHHH^2([0,T],\RR^p)$ or $\SSSS^2([0,T],\RR^p)$.
}
\item[\textup{(L1.2)}] \parbox[t][][l]{\mylength}{ $\LLL$ is bounded and Lipschitz continuous, i.e. there exists a constant $K > 0$ independent of $T$ such that, for all $M, M' \in \MMMM^2([0,T],\RR^d)$,
\begin{align*}
\|\LLL(M)\|_{\OOOO^2[0,T]} &\le K \|M\|_{\SSSS^2[0,T]}, \\
\|\LLL(M)-\LLL(M')\|_{\OOOO^2[0,T]} &\le K\|M-M'\|_{\SSSS^2[0,T]}.
\end{align*}
}
\end{itemize}
}

We will see in \autoref{exun_fbsde} that Assumption {\bfseries (L1)}
is enough to guarantee the local solvability of our system without a
concrete specification of the operators $\LLL^i$. In particular, as
anticipated in the Introduction, the weakness of Assumption
{\bfseries (L1)} allows to consider many different types of
operators within our framework: we emphasize its generality by
giving several examples of possible operators, and we present
potential financial applications.

\begin{expls}\label{expl_mart_repr}
We begin with the classical case of integrand processes generated by
martingale representations.
\begin{enumerate}

\item Assume that $(\FFF_t)_{t \in [0,T]}$ is the augmented filtration generated by the Brownian motion $W$, and take $\OOOO^2([0,T],\RR^p)=\HHHH^2([0,T],\RR^{d \times m})$. Then, we define $\LLL:\MMMM^2([0,T],\RR^d) \to \HHHH^2([0,T],\RR^{d \times m})$ implicitly via It\^o's representation theorem by
\[
M_t=M_0+ \int_0^t \LLL(M)_s d W_s, \quad t \in [0,T].
\]
By It\^o's isometry we have that
\begin{align*}
\|\LLL(M)\|^2_{\HHHH^2[0,T]}&=E\bigg[\int_0^T |\LLL(M)_t|^2 d t\bigg] \\
&=E\big[\big(M_T - M_0\big)^2\!\;\big]=E[M_T^2] - E[M_0^2] \le
\|M\|^2_{\SSSS^2[0,T]},
\end{align*}
and the Lipschitz property follows by the linearity of $\LLL$. Note
that in the case where $\LLL^1=\LLL^2=\LLL^3=\LLL$, the system
\eqref{fbsdyn} is reduced to a classical FBSDE: this shows that
classical FBSDEs can be seen as a special case of the
forward-backward stochastic dynamics \eqref{fbsdyn}.

\item Let now $(\FFF_t)_{t \in [0,T]}$ be a general filtration with just the usual assumptions. As in the previous case, we can take $\OOOO^2([0,T],\RR^p)= \HHHH^2([0,T],\RR^{d \times m})$, and define $\LLL:\MMMM^2([0,T],\RR^d) \to \HHHH^2([0,T],\RR^{d \times m})$ implicitly via the orthogonal decomposition with respect to $W$, i.e.
\[
M_t=\int_0^t \LLL(M)_s d W_s + N_t, \quad t \in [0,T],
\]
where $N$ is some martingale orthogonal with respect to $W$. The
reader may easily notice the connection between this operator and
generalized BSDEs (see \cite{ElKMaz}), and we can prove similarly to
Example \ref{expl_mart_repr} \textup{(}\textit{i}\textup{)} that
$\LLL$ satisfies \textup{{\bfseries (L1)}}, by applying the
Burkholder-Davis-Gundy inequality instead of It\^o's isometry.

We can thus study generalized fully coupled FBSDEs within our
framework, and the choice of the filtration allows us to consider,
in typical financial applications, the case of incomplete markets.
An illustrative example is that of a large investor trading in an
incomplete market: since this investor buys and sells large amounts
of assets, it is reasonable to assume that his trading strategy
affects the prices of the stocks. By considering the corresponding
hedging problem, we thus obtain a fully coupled system, and the
incompleteness of the market leads to a generalized fully coupled
FBSDE. For more details, we refer the reader to \cite{MaYon}.
\end{enumerate}
\end{expls}

\begin{expls}\label{non_class_expls}
While martingale integrand processes are the case most studied in
the literature, they are not the only class of operators fitting in
our framework: there are indeed several other classes of non-local
operators, not considered in the classical FBSDE literature, which
satisfy Assumption \textup{{\bfseries (L1)}}. Let us give some
examples.
\begin{enumerate}

\item Assume that $(\FFF_t)_{t \in [0,T]}$ just satisfies the usual assumptions. We take $\OOOO^2([0,T],\RR^p)= \SSSS^2([0,T],\RR^d)$, and $\LLL:\MMMM^2([0,T],\RR^d) \to \SSSS^2([0,T],\RR^d)$ is simply defined by $\LLL(M):=M$; in this case, Assumption \textup{{\bfseries (L1)}} becomes trivial. In a financial context, $\LLL(M)$ may represent the risky part of the claim $Y=M - V$.

\item Let for simplicity $d=1$. Fix $\widetilde{T} > 0$ and assume that, for $T \le \widetilde{T}$, $\FF=(\FFF_t)_{t \in [0,T]}$ is such that all martingales with respect to $\FF$ are continuous. Let $\OOOO^2([0,T],\RR^p) =\HHHH^2([0,T],\RR)$, and define $\LLL:\MMMM^2([0,T],\RR) \to \HHHH^2([0,T],\RR)$ by
\[
\LLL(M)_t:=\sqrt{E\big[\langle M \rangle_{t,T} \big|\FFF_t\big]},
\quad M \in \MMMM^2([0,T],\RR), \; t \in [0,T],
\]
where for notational simplicity $\langle M \rangle_{t,T}:=\langle M
\rangle_T - \langle M \rangle_t$. Then, by the Kunita-Watanabe and
the conditional Cauchy-Schwarz inequalities, we have that
\begin{align*}
E\big[\langle M, M' \rangle_{t,T} \big|\FFF_t\big] \!\le\!
E\big[|\langle M, M' \rangle_{t,T}| \big|\FFF_t\big]
&\!\le\! E\Big[\sqrt{\langle M \rangle_{t,T} \langle M' \rangle_{t,T}} \Big|\FFF_t\Big] \\
&\!\le\! \sqrt{E[\langle M \rangle_{t,T} |\FFF_t]} \sqrt{E[\langle
M' \rangle_{t,T} |\FFF_t]},
\end{align*}
and therefore, by the bilinearity of $\langle \cdot \rangle_{t,T}$,
\begin{align*}
\big|\LLL(M)_t - \LLL(M')_t \big|^2
&= E[\langle M \rangle_{t,T} |\FFF_t] + E[\langle M' \rangle_{t,T} |\FFF_t] \\
&\phantom{=}- 2 \sqrt{E[\langle M \rangle_{t,T} |\FFF_t]} \sqrt{E[\langle M' \rangle_{t,T} |\FFF_t]} \\
&\le E[\langle M - M' \rangle_{t,T} |\FFF_t]
\end{align*}
for all $t \ge 0$. Hence, by Fubini's theorem,
\begin{align*}
\|\LLL(M) - \LLL(M')\|_{\HHHH^2[0,T]}^2&=\!E\bigg[\int_0^T \!\big|\LLL(M)_t - \LLL(M')_t \big|^2 d t \bigg] \\
&\le\! E\bigg[\int_0^T \!E\big[\langle M \!-\! M' \rangle_{t,T}\big|\FFF_t \big] d t \bigg] \\
&=\! E\bigg[\int_0^T \!\langle M \!-\! M' \rangle_{t,T} d t \bigg]
\le \! \widetilde{T} E[\langle M \!-\! M' \rangle_T].
\end{align*}
By applying the Burkholder-Davis-Gundy inequality, we finally get
the desired Lipschitz property. The boundedness condition is
obtained via similar computations.

\item We choose again for simplicity $d=1$, and we assume that $\FF=(\FFF_t)_{t \in [0,T]}$ is as in Example \ref{non_class_expls} \textup{(}\textit{ii}\textup{)}. We can define $\LLL$ by taking $\OOOO^2([0,T],\RR^p) =\SSSS^2([0,T],\RR)$, and
\[
\LLL:\MMMM^2([0,T],\RR) \to \SSSS^2([0,T],\RR), \quad
\LLL(M)_t:=\sqrt{\langle M \rangle_t}, \; t \in [0,T].
\]
By the Burkholder-Davis-Gundy inequality, we have that
\[
\|\LLL(M)\|^2_{\SSSS^2[0,T]}=E\Big[\sup_{t \in [0,T]} \langle M
\rangle_t\Big]=E[\langle M \rangle_T ] \le K \|M\|^2_{\SSSS^2[0,T]}.
\]
On the other hand, by applying the Kunita-Watanabe inequality,
\begin{align*}
\Big|\sqrt{\langle M \rangle_t} - \sqrt{\langle M' \rangle_t}\Big|^2
&= \langle M \rangle_t + \langle M' \rangle_t - 2 \sqrt{\langle M \rangle_t \langle M' \rangle_t} \\
&\le \langle M \rangle_t + \langle M' \rangle_t - 2 \big|\langle M , M' \rangle_t\big| \\
&\le \langle M \rangle_t + \langle M' \rangle_t - 2 \langle M , M'
\rangle_t =\langle M - M' \rangle_t,
\end{align*}
and the Lipschitz property then follows by applying the
Burkholder-Davis-Gundy inequality as above. We restrict for a moment
to the Brownian setting to give a financial interpretation: in this
case, $\LLL(M)$ can be explicitly rewritten as $\LLL(M)_t=
\sqrt{\int_0^t |Z_s|^2 d s}$, where $Z$ is the martingale integrand
in the It\^o representation of $M$. In the usual BSDE framework for
hedging (see for instance \cite{ElKPenQue}), $\LLL(M)$ is then
closely connected to the accumulated cost of the portfolio strategy:
this could allow us, for instance, to consider storage problems
within our setting.

\item We modify the previous example by combining it with orthogonal decompositions. Let $d=1$ and $\FF=(\FFF_t)_{t \in [0,T]}$ as in Example \ref{non_class_expls} \textup{(}\textit{ii}\textup{)}. Fix a martingale $\widetilde{M}$, and define $\RRR:\MMMM^2([0,T],\RR) \to \MMMM^2([0,T],\RR)$
 as the orthogonal term in the orthogonal decomposition with respect to $\widetilde{M}$, i.e.
\[
M_t=\big(M_t - \RRR(M)_t\big) + \RRR(M)_t, \quad t \in [0,T],
\]
where $\RRR(M)$ is orthogonal with respect to $\widetilde{M}$ and
$M_t - \RRR(M)_t=\int_0^t Z_s d \widetilde{M}_s$ for some process
$Z$. Then, we define the operator $\LLL$ by
\[\LLL:\MMMM^2([0,T],\RR) \to \SSSS^2([0,T],\RR), \quad
\LLL(M)_t:=\sqrt{\langle \RRR(M) \rangle_t}, \; t \in [0,T].
\]
Because of the orthogonality, it is easy to check that, for all $t
\ge 0$,
\begin{align*}
\langle M \rangle_t &= \langle M - \RRR(M) + \RRR(M) \rangle_t \\
&= \langle M - \RRR(M) \rangle_t + \langle \RRR(M) \rangle_t \ge
\langle \RRR(M) \rangle_t,
\end{align*}
and similarly for $M-M'$. Therefore, $\LLL$ satisfies Assumption
{\bfseries (L1)} because of Example \ref{non_class_expls}
\textup{(}\textit{iii}\textup{)}. This operator may have interesting
applications in mathematical finance: namely, in the typical BSDE
framework for hedging in incomplete markets, the operator $\RRR(M)$
represents the non-hedgeable part of the claim. While we cannot
hedge this risk, we can partly incorporate its effect on the price
process since the coefficients $\mu$, $\sigma$ and $f$ are allowed
to depend on $\sqrt{\langle \RRR(M) \rangle_t}$.

\item As a final example, we introduce an operator intimately related to backward equations with time delayed generators (see \cite{DelImk}). Let $\widetilde{T} > 0$ be fixed. For $T \le \widetilde{T}$, let $(\FFF_t)_{t \in [0,T]}$ satisfy the usual assumptions, and $\OOOO^2([0,T],\RR^p)= \HHHH^2([0,T],\RR^{d \times m})$. Motivated by the framework introduced by Dos Reis et al. \cite{DosRevZha}, we can then define $\LLL:\MMMM^2([0,T],\RR^d) \to \HHHH^2([0,T],\RR^{d \times m})$ by
\[
\LLL(M)_t:=\int_{-t}^0 \widehat{\LLL}(M)_{t+s} \alpha_Z (d s),
\]
where $\widehat{\LLL}$ is the operator introduced in Example
\ref{expl_mart_repr} \textup{(}\textit{ii}\textup{)}, and $\alpha_Z$
is a non-random finite measure with support in $[-\widetilde{T},0]$.
We observe that $\LLL$ is closely related to the operator of Example
\ref{non_class_expls} \textup{(}\textit{iii}\textup{)} when
$\alpha_Z$ is the Lebesgue measure restricted to
$[-\widetilde{T},0]$. Moreover, by applying the change of
integration order proved in \cite{DosRevZha}, we can show that
\begin{equation}\label{change_order}
\begin{split}
\|\LLL(M)\|_{\HHHH^2[0,T]}&=E\bigg[\int_0^T \bigg|\int_{-t}^0 \widehat{\LLL}(M)_{t+s} \alpha_Z (d s)\bigg|^2 d t\bigg]^{1/2} \\
&\le \alpha_Z ([-\widetilde{T},0])
\|\widehat{\LLL}(M)\|_{\HHHH^2[0,T]} \le K \|M\|_{\SSSS^2[0,T]}
\end{split}
\end{equation}
for some constant $K=K(\widetilde{T})$, where the last inequality
follows from the boundedness of $\widehat{\LLL}$. Assumption
\textup{{\bfseries (L1)}} then follows by linearity: this will
therefore allow us to consider a special class of fully coupled
forward-backward equations with delayed generators.

However, the treatment of backward stochastic delayed equations in
their full generality would require to replace $\YYY(X,V)$ in
\eqref{fbfunctdiffeq} by a new operator $\YY(X,V)_t:=\int_{-t}^0
\YYY(X,V)_{t+s} \alpha_Y (d s)$, where $\alpha_Y$ is a non-random
finite measure: as pointed out in \cite{DosRevZha}, this kind of
equation is very difficult to study even in the simple decoupled
case, and it will not be treated here. For an overview of the
several applications of backward equations with time delay to
mathematical finance, the reader may consult for instance the
article by Delong \cite{Delong}.

\end{enumerate}
\end{expls}

The broad generality of the above examples should help us understand
the importance of the next theorem. This is the main result of this
section, and gives the existence of a unique square-integrable
solution to the functional differential system \eqref{fbfunctdiffeq}
on sufficiently small intervals $[0,T]$, provided that the operators
$\LLL^i$ satisfy Assumption {\bfseries (L1)}. This gives us a great
flexibility in the choice of $\LLL^i$, allowing to study many
different types of fully coupled, non-classical forward-backward
stochastic dynamics.

\begin{theorem}\label{exun_fbsde}
Let $\mu$, $\sigma$, $f$ and $\phi$ satisfy Assumption
\textup{{\bfseries (A1)}} with respect to the constant $C$.
Furthermore, assume that $\LLL^1$, $\LLL^3$ satisfy Assumption
\textup{{\bfseries (L1)}} and that $\LLL^2$ satisfies
\textup{{\bfseries (L1)}} with respect to $\OOOO^2([0,T],\RR^{p_2})
=\SSSS^2([0,T],\RR^{p_2})$, denoting by $K$ the common Lipschitz
constant of $\LLL^1$, $\LLL^2$, $\LLL^3$. Then there is a constant
$\ell=\ell(C,K)$ depending only on $C$ and $K$ so that, for $T <
\ell$, \eqref{fbfunctdiffeq} admits a unique solution $(X,V)$ in
$\SSSS^2([0,T],\RR^n) \times \SSSS^2([0,T],\RR^d)$.
\end{theorem}

\begin{proof}
From now on, we write $\|\cdot\|_\OOOO$ for the norm
$\|\cdot\|_{\OOOO^2[0,T]}$ for notational simplicity. Moreover, we
denote the product space $\SSSS^2([0,T],\RR^n) \times
\SSSS^2([0,T],\RR^d)$ by $\SSSS^2_X \times \SSSS^2_V$, and endow it
with the norm
\[
\|(X,V)\|_{\SSSS^2_X \times
\SSSS^2_V}:=\sqrt{\|X\|_{\SSSS^2}^2+\|V\|_{\SSSS^2}^2}, \quad (X,V)
\in \SSSS^2_X \times \SSSS^2_V.
\]
The mapping $\LL:\SSSS^2_X \times \SSSS^2_V \rightarrow \SSSS^2_X
\times \SSSS^2_V$, $\LL(X,V):=({\widetilde X},{\widetilde V})$, is
defined as follows: first, ${\widetilde X}$ is constructed as the
unique solution in $\SSSS^2([0,T],\RR^n)$ to the forward stochastic
differential equation \small
\begin{numcases}{\hspace{-0.2em}}
d {\widetilde X}_t= \mu(t,{\widetilde X}_t, \YYY(X\!, V)_t, \LLL^1(\MMM(X\!, V))_t) d t \!+\! \sigma(t,{\widetilde X}_t,\YYY(X\!, V)_t,\LLL^2(\MMM(X\!, V))_t) d W_t, \nonumber \\
{\widetilde X}_0=x. \label{sdeX}
\end{numcases}
\normalsize Then, once ${\widetilde X}$ has been obtained,
${\widetilde V}$ is given explicitly by the expression
\begin{equation}\label{sdeV}
{\widetilde V}_t=\int_0^t f(s,{\widetilde X}_s, \YYY(X,V)_s,
\LLL^3(\MMM(X,V))_s) d s.
\end{equation}

We show that the mapping $\LL$ is well defined and maps $\SSSS^2_X
\times \SSSS^2_V$ into itself. First of all, we note that the
existence of a unique solution in $\SSSS^2([0,T],\RR^n)$ to
\eqref{sdeX} follows by Assumption \textbf{(A1)} and the results on
stochastic differential equations with monotonous coefficients
obtained, for instance, by Rozovsky \cite{Roz}: indeed, by setting
$\bar{\mu}(t,x):=\mu(t,x, \YYY(X, V)_t, \LLL^1(\MMM(X, V))_t)$ and
$\bar{\sigma}(t,x):=\sigma(t,x, \YYY(X, V)_t,\LLL^2(\MMM(X, V))_t)$,
the reader can easily verify that all the conditions of \cite{Roz}
are satisfied, as $\|\YYY(X,V)\|_{\SSSS^2}$ and
$\|\MMM(X,V)\|_{\SSSS^2}$ are finite by \autoref{fbest_y&m}. On the
other hand, it is easy to verify that $\widetilde{V} \in
\SSSS^2([0,T],\RR^d)$, as for $f_0:=f(\cdot,0,0,0)$,
\begin{align*}
&\|\widetilde{V}\|_{\SSSS^2} \le  \sqrt{T} \Big(\|f_0\|_{\HHHH^2}+ \|f\big(\cdot,{\widetilde X}_\cdot,\YYY(X,V)_\cdot,\LLL^3(\MMM(X,V))_\cdot\big)- f_0\|_{\HHHH^2}\Big) \\
&\!\le \!\! \sqrt{T}\Big(\|f_0\|_{\HHHH^2}+ C (\sqrt{T} \vee 1) \big( \|\widetilde{X}\|_{\SSSS^2}+\|\YYY(X, V)\|_{\SSSS^2}+\|\LLL^3(\MMM(X, V))\|_{\OOOO^2}\big)\Big) \\
&\!\le \!\! \sqrt{T}\Big(\!\|f_0\|_{\HHHH^2}\!+ C (\sqrt{T} \vee 1)
\big(\|\widetilde{X}\|_{\SSSS^2}\!+\!\|\YYY(X,V)\|_{\SSSS^2}\!+\!K
\|\MMM(X,V)\|_{\SSSS^2}\big)\!\Big)\!<\! \infty.
\end{align*}

Since the pair $(X,V)$ is a solution of \eqref{fbfunctdiffeq} if and
only if it is a fixed point of $\LL$, it suffices to prove that
$\LL$ is a contraction on $\SSSS^2_X \times \SSSS^2_V$ for small
enough $T > 0$. Let $\LL(X^1,V^1)=({\widetilde X}^1,{\widetilde
V}^1)$, $\LL(X^2,V^2)=({\widetilde X}^2,{\widetilde V}^2)$, and
assume without loss of generality that $T \le 1$. By It\^o's
formula, we can compute that
\begin{align*}
&d|{\widetilde X}^1_t - {\widetilde X}^2_t|^2 = 2 ({\widetilde X}^1_t - {\widetilde X}^2_t)^{\mathrm T} d({\widetilde X}^1 - {\widetilde X}^2)_t + d \langle {\widetilde X}^1 - {\widetilde X}^2 \rangle_t \\
&=2\big({\widetilde X}^1_t - {\widetilde X}^2_t\big)^{\mathrm T} \Big(\mu\big(t,{\widetilde X}^1_t, \YYY(X^1,V^1)_t, \LLL^1(\MMM(X^1,V^1))_t\big) \\
&\phantom{=2\big({\widetilde X}^1_t - {\widetilde X}^2_t\big)^{\mathrm T} \big)}- \mu\big(t,{\widetilde X}^2_t, \YYY(X^2,V^2)_t, \LLL^1(\MMM(X^2,V^2))_t\big) \Big) d t \\
&\phantom{=}+ 2 \big({\widetilde X}^1_t-{\widetilde X}^2_t\big)^{\mathrm T} \Big(\sigma\big(t, {\widetilde X}^1_t,\YYY(X^1,V^1)_t,\LLL^2(\MMM(X^1,V^1))_t\big) \\
&\phantom{=2\big({\widetilde X}^1_t - {\widetilde X}^2_t\big)^{\mathrm T} \big)}-\sigma\big(t,{\widetilde X}^2_t,\YYY(X^2,V^2)_t,\LLL^2(\MMM(X^2,V^2))_t\big)\Big) d W_t \\
&\phantom{=}+\big|\sigma\big(t, {\widetilde X}^1_t,\YYY(X^1,V^1)_t,\LLL^2(\MMM(X^1,V^1))_t\big) \\
&\phantom{= +\big|} - \sigma\big(t,{\widetilde
X}^2_t,\YYY(X^2,V^2)_t,\LLL^2(\MMM(X^2,V^2))_t\big)\big|^2 d t
\end{align*}
for any $t \ge 0$. Thus, by applying Assumption (A1.2), (A1.4) and
classical inequalities we obtain that, for a constant
$\theta_1=\theta_1(C)$ depending only on $C$,
\begin{align}
&\|{\widetilde X}^1 - {\widetilde X}^2\|_{\SSSS^2}^2=E\bigg[\sup_{t \in [0,T]} |{\widetilde X}^1_t - {\widetilde X}^2_t|^2 \bigg] \nonumber \\
&\le \theta_1 \Bigg(\!E\bigg[\int_0^T\!\! |{\widetilde X}^1_s -
{\widetilde X}^2_s|^2 d s \bigg]
 \!+\! E\bigg[\int_0^T \!\!|{\widetilde X}^1_s - {\widetilde X}^2_s| |\YYY(X^1,V^1)_s \!-\! \YYY(X^2,V^2)_s| d s \bigg] \nonumber \\
&\phantom{=}+ E\bigg[\int_0^T |{\widetilde X}^1_s - {\widetilde X}^2_s|  |\LLL^1(\MMM(X^1,V^1))_s - \LLL^1(\MMM(X^2,V^2))_s| d s \bigg] \nonumber \\
&\phantom{=}+ E\bigg[\int_0^T |\YYY(X^1,V^1)_s - \YYY(X^2,V^2)_s|^2 d s \bigg] \nonumber \\
&\phantom{=}+ E\bigg[\int_0^T |\LLL^2(\MMM(X^1,V^1))_s - \LLL^2(\MMM(X^2,V^2))_s|^2 d s \bigg] \nonumber\\
&\phantom{=}+E\bigg[\bigg(\!\int_0^T\!\! |{\widetilde X}^1_s - {\widetilde X}^2_s|^2 \big( |{\widetilde X}^1_s - {\widetilde X}^2_s|^2 \! + \!  |\YYY(X^1,V^1)_s - \YYY(X^2,V^2)_s|^2 \nonumber \\
&\phantom{=+E\bigg[\bigg(\!\int_0^T\!\! |{\widetilde X}^1_s - } + |\LLL^2(\MMM(X^1,V^1))_s - \LLL^2(\MMM(X^2,V^2))_s|^2 \big) d s \bigg)^{1/2} \bigg]\Bigg)\nonumber \\
&=: \theta_1 \Big(I_1+I_2+I_3+I_4+I_5+I_6\Big). \label{MainEst}
\end{align}
\indent The next step consists in deriving estimates for all the
terms on the right hand side of this inequality. $I_1$ can simply be
estimated by
\begin{equation} \label{Est0}
I_1 \le T \|{\widetilde X}^1 - {\widetilde X}^2 \|^2_{\SSSS^2}.
\end{equation}
For $I_2$, we obtain by the Cauchy-Schwarz inequality that
\[
I_2 \le \frac{T}{2} \Big( \|{\widetilde X}^1 - {\widetilde
X}^2\|^2_{\SSSS^2} + \|\YYY(X^1,V^1) - \YYY(X^2,V^2)\|^2_{\SSSS^2}
\Big).
\]
Because of assumption (A1.4) on $\phi$, we can apply
\autoref{fbest_y&m} and get that
\begin{equation}
I_2 \le \theta_2 T \Big( \|{\widetilde X}^1 - {\widetilde
X}^2\|^2_{\SSSS^2} + \|V^1 - V^2\|^2_{\SSSS^2} + \|X^1 -
X^2\|^2_{\SSSS^2} \Big), \label{Est1}
\end{equation}
for some constant $\theta_2=\theta_2(C)$ depending on $C$. To
estimate $I_3$, one can apply again the Cauchy-Schwarz inequality to
get that
\begin{align*}
I_3 &\le \frac{T}{2 \eps} \|{\widetilde X}^1 - {\widetilde X}^2\|^2_{\SSSS^2} + \frac{\eps}{2} E\bigg[\int_0^T |\LLL^1(\MMM(X^1,V^1))_s - \LLL^1(\MMM(X^2,V^2))_s|^2 d s \bigg] \\
&\le \frac{\sqrt{T}}{2} \Big( \|{\widetilde X}^1 - {\widetilde
X}^2\|^2_{\SSSS^2} + \|\LLL^1(\MMM(X^1,V^1)) -
\LLL^1(\MMM(X^2,V^2))\|^2_{\OOOO^2} \Big),
\end{align*}
where the last inequality is obtained by taking $\eps=\sqrt{T}$ if
$\OOOO^2=\HHHH^2$ and $\eps=1$ if $\OOOO^2=\SSSS^2$ (remember that
$T \le 1$). On the other hand, by Assumption \textbf{(L1)} we know
that
\[
\|\LLL^1(\MMM(X^1,V^1)) - \LLL^1(\MMM(X^2,V^2))\|^2_{\OOOO^2} \le
K^2 \|\MMM(X^1,V^1) - \MMM(X^2,V^2)\|^2_{\SSSS^2},
\]
and by applying \autoref{fbest_y&m},
\begin{equation}
I_3 \le \theta_3 \sqrt{T}\Big( \|{\widetilde X}^1 - {\widetilde
X}^2\|^2_{\SSSS^2} + \|V^1 - V^2\|^2_{\SSSS^2} + \|X^1 -
X^2\|^2_{\SSSS^2} \Big), \label{Est2}
\end{equation}
for a constant $\theta_3=\theta_3(C,K)$ depending only on $C$ and
$K$. $I_4$ is estimated by using the same argument as for
\eqref{Est1}, obtaining that
\begin{equation}
I_4 \le \theta_4 T \Big( \|V^1 - V^2\|^2_{\SSSS^2} + \|X^1 -
X^2\|^2_{\SSSS^2} \Big), \label{Est3}
\end{equation}
for some constant $\theta_4=\theta_4(C)$. For $I_5$, we have that
\begin{align*}
I_5 &\le T \|\LLL^2(\MMM(X^1,V^1)) - \LLL^2(\MMM(X^2,V^2))\|^2_{\SSSS^2} \\
&\le K^2 T \|\MMM(X^1,V^1) - \MMM(X^2,V^2)\|^2_{\SSSS^2},
\end{align*}
and hence, by \autoref{fbest_y&m},
\begin{equation}
I_5 \le \theta_5 T \Big( \|V^1 - V^2\|^2_{\SSSS^2} + \|X^1 -
X^2\|^2_{\SSSS^2} \Big), \label{Est4}
\end{equation}
for $\theta_5=\theta_5(K)$. It only remains to estimate the last
term $I_6$: for notational simplicity, we introduce the process
$A^i:=\Big(\YYY(X^i,V^i), \LLL^2(\MMM(X^i,V^i))\Big)$ for $i=1,2$.
By applying the Cauchy-Schwarz inequality, we can verify that
\begin{align*}
I_6&=E\bigg[\bigg(\int_0^T |{\widetilde X}^1_s - {\widetilde X}^2_s|^2 \big( |{\widetilde X}^1_s - {\widetilde X}^2_s|^2 + |A^1_s - A^2_s|^2 \big) d s \bigg)^{1/2} \bigg] \\
&\le E\bigg[\Big(\sup_{t \in [0,T]} |{\widetilde X}^1_t - {\widetilde X}^2_t|^2 \Big)^{1/2} \bigg(\int_0^T \big( |{\widetilde X}^1_s - {\widetilde X}^2_s|^2+|A^1_s - A^2_s|^2 \big) d s \bigg)^{1/2} \bigg] \\
&\le \frac{\sqrt{T}}{2} E\Big[\sup_{t \in [0,T]} |{\widetilde X}^1_t
- {\widetilde X}^2_t|^2 \Big] + \frac{1}{2 \sqrt{T}} E\bigg[\int_0^T
\big( |{\widetilde X}^1_s- {\widetilde X}^2_s|^2+|A^1_s - A^2_s|^2
\big) d s \bigg],
\end{align*}
and it is not difficult to check that
\[
E\bigg[\int_0^T \big( |{\widetilde X}^1_s - {\widetilde
X}^2_s|^2+|A^1_s - A^2_s|^2 \big) d s \bigg] \le T
\Big(\|{\widetilde X}^1 - {\widetilde X}^2\|^2_{\SSSS^2} + \|A^1 -
A^2\|^2_{\SSSS^2}\Big),
\]
which finally leads us to
\begin{equation}
I_6 \le \theta_6 \sqrt{T} \Big(\|{\widetilde X}^1 - {\widetilde
X}^2\|^2_{\SSSS^2} + \|V^1 - V^2\|^2_{\SSSS^2} + \|X^1 -
X^2\|^2_{\SSSS^2} \Big), \label{Est5}
\end{equation}
for some constant $\theta_6=\theta_6(C)$. Therefore, we can plug the
estimates \eqref{Est0}--\eqref{Est5} back into \eqref{MainEst},
obtaining the inequality
\begin{equation}
\|{\widetilde X}^1 - {\widetilde X}^2\|_{\SSSS^2}^2 \le \theta_7
\sqrt{T} \Big(\|{\widetilde X}^1 - {\widetilde X}^2\|^2_{\SSSS^2} +
\|V^1 - V^2\|^2_{\SSSS^2} + \|X^1 - X^2\|^2_{\SSSS^2} \Big),
\label{Est_X}
\end{equation}
for some constant $\theta_7=\theta_7(C,K)$. On the other hand,
thanks to the explicit nature of the functional differential
equation \eqref{sdeV}, we can estimate $\|{\widetilde V}^1 -
{\widetilde V}^2\|_{\SSSS^2}$ as follows:
\begin{align*}
\|{\widetilde V}^1 &- {\widetilde V}^2\|_{\SSSS^2}^2
\le E\bigg[\bigg(\int_0^T \Big| f\big(s,{\widetilde X}^1_s, \YYY(X^1,V^1)_s, \LLL^3(\MMM(X^1,V^1))_s\big) \\
&\phantom{{\widetilde V}^2\|_{\SSSS^2}^2
\le E\bigg[\bigg(} - f\big(s,{\widetilde X}^2_s, \YYY(X^2,V^2)_s, \LLL^3(\MMM(X^2,V^2))_s\big) \Big| d s\bigg)^2\bigg] \\
&\le T E\bigg[\int_0^T \Big| f\big(s,{\widetilde X}^1_s, \YYY(X^1,V^1)_s, \LLL^3(\MMM(X^1,V^1))_s\big) \\
&\phantom{\le E\bigg[\bigg(} - f\big(s,{\widetilde X}^2_s, \YYY(X^2,V^2)_s, \LLL^3(\MMM(X^2,V^2))_s\big) \Big|^2 d s\bigg]  \\
&\le 3 C^2 T \bigg( E\bigg[\int_0^T\!\big|{\widetilde X}^1_s - {\widetilde X}^2_s\big|^2 d s\bigg] + E\bigg[\int_0^T \!\big|\YYY(X^1\!,V^1)_s - \YYY(X^2\!,V^2)_s\big|^2 d s \bigg] \\
&\phantom{\le} + E\bigg[\int_0^T \big|\LLL^3(\MMM(X^1,V^1))_s - \LLL^3(\MMM(X^2,V^2))_s\big|^2 d s \bigg] \bigg) \\
&\le 3 C^2 T \Big( \|{\widetilde X}^1 - {\widetilde X}^2\|_{\SSSS^2}^2 + \|\YYY(X^1,V^1) - \YYY(X^2,V^2)\|_{\SSSS^2}^2 \\
&\phantom{\le 3 C^2 T \Big( } + \|\LLL^3(\MMM(X^1,V^1)) -
\LLL^3(\MMM(X^2,V^2))\|_{\OOOO^2}^2 \Big),
\end{align*}
and by using the same arguments as for the estimates
\eqref{Est0}--\eqref{Est5} this yields that, for a constant
$\theta_8=\theta_8(C,K)$,
\begin{equation}
\|{\widetilde V}^1 - {\widetilde V}^2\|_{\SSSS^2}^2 \le \theta_8
\sqrt{T} \Big(\|{\widetilde X}^1 - {\widetilde X}^2\|_{\SSSS^2}^2 +
\|V^1 - V^2\|_{\SSSS^2}^2 + \|X^1_t - X^2_t\|_{\SSSS^2}^2 \Big).
\label{Est_V}
\end{equation}
Hence, we can sum the inequalities \eqref{Est_X} and \eqref{Est_V},
obtaining a constant $\theta_9=\theta_9(C,K)$ depending only on $C$
and $K$ such that
\begin{multline*}
\|({\widetilde X}^1,{\widetilde V}^1) - ({\widetilde X}^2,{\widetilde V}^2)\|_{\SSSS^2_X \times \SSSS^2_V}^2 \le \theta_9 \sqrt{T} \Big(\|{\widetilde X}^1 - {\widetilde X}^2\|_{\SSSS^2}^2 +\|V^1 - V^2\|_{\SSSS^2}^2 +\|X^1 - X^2\|_{\SSSS^2}^2 \Big) \\
\le \theta_9 \sqrt{T} \Big(\|({\widetilde X}^1,{\widetilde V}^1) -
({\widetilde X}^2,{\widetilde V}^2)\|_{\SSSS^2_X \times \SSSS^2_V}^2
+ \|(X^1,V^1) - (X^2,V^2)\|_{\SSSS^2_X \times \SSSS^2_V}^2 \Big),
\end{multline*}
which implies that, for $T>0$ such that $\theta_9 \sqrt{T} < 1/2$,
\[
\|({\widetilde X}^1,{\widetilde V}^1) - ({\widetilde
X}^2,{\widetilde V}^2)\|_{\SSSS^2_X \times \SSSS^2_V}^2 \le
\underbrace{\bigg(\frac{1}{\theta_9 \sqrt{T}} -1\bigg)^{-1}}_{< 1}
\|(X^1,V^1) - (X^2,V^2)\|_{\SSSS^2_X \times \SSSS^2_V}^2.
\]
Therefore, $\LL$ is a contraction if $T< \ell(C,K):=\frac{1}{4
\theta_9^2(C,K)} \wedge 1$, and thus admits a unique fixed point
$(X,V)$.
\end{proof}

\begin{rem}
Under the general assumption \textup{{\bfseries (A1)}}, it is not
possible to extend \autoref{exun_fbsde} to the case where $\LLL^2$
satisfies {\bfseries (L1)} with respect to $\OOOO^2([0,T],\RR^{p_2})
=\HHHH^2([0,T],\RR^{p_2})$, as shown by the following counterexample
borrowed from the theory of FBSDEs. Assume we have an augmented
Brownian filtration, and consider the functional differential system
\[
\begin{cases}
d X_t = \LLL^2(\MMM^\phi(X,V))_t d W_t, \\
d V_t = 0, \\
X_0=V_0=0,
\end{cases}
\]
where $\LLL^2$ is the operator given by It\^o's representation, and
$\phi(x):=x + W_T$. We thus obtain that $V \equiv 0$, and the first
equation can be rewritten as
\[
d X_t = d \MMM^\phi(X,0)_t, \quad X_0=0.
\]
Assume it has an adapted solution $X$. We would then have
$X_t=E[X_T+W_T|\FFF_t]$ for all $t \ge 0$, which would lead in
particular, for $t=T$, to the contradiction $W_T=0$.
\end{rem}

\begin{rem}\label{rem_initialtime}
While already quite general, \autoref{exun_fbsde} can be further
extended in a number of different ways. We present some possible
extensions which only require slight modifications of the proof
presented above: the details are left to the reader.
\begin{enumerate}
\item First of all, \autoref{exun_fbsde} can be extended to any initial time $\tau >0$. More exactly, consider functional differential systems on $[\tau,T]$ of the form
\small
\[
\hspace{-0.5em}\begin{cases}
d X_t= \mu(t,X_t, \YYY(X,V)_t, \LLL^1(\MMM(X,V))_t) d t + \sigma(t,X_t,\YYY(X,V)_t, \LLL^2(\MMM(X,V))_t) d W_t, \\
d V_t=f(t,X_t, \YYY(X,V)_t, \LLL^3(\MMM(X,V))_t) d t, \\
X_\tau=\eta, \quad V_\tau=\zeta,
\end{cases}
\]
\normalsize where $\eta, \zeta \in L^2(\FFF_\tau)$. Then it is
possible to prove that, under the same conditions of
\autoref{exun_fbsde} and provided that $T-\tau < \ell$, there is a
unique solution in $\SSSS^2([\tau,T],\RR^n) \times
\SSSS^2([\tau,T],\RR^d)$. This remark will be essential in the next
section, where we extend the solvability to any time interval.

\item It is not difficult to see that \autoref{exun_fbsde} also holds for operators $\LLL^1$ and $\LLL^3$ of the form $\LLL^i=(\LLL^i_1, \cdots, \LLL^i_{k_i}, \LLL^i_{k_i+1}, \dots, \LLL^i_{l_i})$, $i=1,3$, where $\LLL^i_j$ satisfies Assumption {\bfseries (L1)} with respect to $\SSSS^2\big([0,T],\RR^{p^i_j}\big)$ for $1 \le j \le k_i$, and with respect to $\HHHH^2\big([0,T],\RR^{p^i_j}\big)$ for $k_i+1 \le j \le l_i$: indeed, it suffices to consider the space $\OOOO^2\big([0,T],\RR^{p^i}\big):=\prod_{j=1}^{k_i} \SSSS^2\big([0,T],\RR^{p^i_j}\big) \times \prod_{j=k_i+1}^{l_i} \HHHH^2\big([0,T],\RR^{p^i_j}\big)$ and take the norm given by
\[
\|L\|^2_{\OOOO^2([0,T],\RR^{p^i})}:= \sum_{j=1}^{k_i}
\|L\|^2_{\SSSS^2([0,T],\RR^{p^i_j})} + \sum_{j=k_i+1}^{l_i}
\|L\|^2_{\HHHH^2([0,T],\RR^{p^i_j})}.
\]

\item The result of \autoref{exun_fbsde} can also be extended to other types of terminal condition. Indeed, assume that $\DD([0,T],\RR^n)$ denotes the space of all $\RR^n$-valued c\`adl\`ag functions on $[0,T]$, and let $\Phi: \Omega \times \DD([0,T],\RR^n) \rightarrow \RR^d$ satisfy the $L^\infty$-Lipschitz condition
\[
|\Phi(\mathbf{x}) - \Phi(\mathbf{x}')| \le C \sup_{t \in [0,T]}
|\mathbf{x}_t - \mathbf{x}'_t| \quad P\text{-a.s.} \quad \forall \,
\mathbf{x},\mathbf{x}' \in \DD([0,T],\RR^n),
\]
where $C > 0$. Then, the operators
\[
\overline{\!\MMM}(X,V)_t:=E[\Phi(X) + V_T|\FFF_t], \quad
\overline{\YYY}(X,V)_t:=\overline{\!\MMM}(X,V)_t - V_t,
\]
satisfy estimates similar to those of \autoref{fbest_y&m}.
Therefore, the reader can easily verify that \autoref{exun_fbsde}
remains valid if we substitute the operators $\MMM$ and $\YYY$ in
the system \eqref{fbfunctdiffeq} by $\overline{\!\MMM}$ and
$\overline{\YYY}$. Moreover, the same conclusion remains true if
$\Phi$ satisfies the $L^1$-Lipschitz condition
\[
|\Phi(\mathbf{x}) - \Phi(\mathbf{x}')| \le C \int_0^T |\mathbf{x}_t
- \mathbf{x}'_t| d t \quad P\text{-a.s.} \quad \forall \,
\mathbf{x},\mathbf{x}' \in \DD([0,T],\RR^n),
\]
instead of the above $L^\infty$-Lipschitz condition. Two typical
examples are the functionals $\Phi^1(\mathbf{x})=\sup_{t \in [0,T]}
|\mathbf{x}_t|$ and $\Phi^2(\mathbf{x})=\int_0^T \mathbf{x}_t d t$,
which are related to lookback and Asian options.

\item Finally, it is possible to generalize the system \eqref{fbfunctdiffeq} by substituting the equation for the component $V$ by
\[
d V_t=f(t,X_t, \YYY(X,V)_t, \LLL^3(\MMM(X,V))_t) \alpha_V(d t),
\quad V_0=0,
\]
where, for some $\widetilde{T} > 0$, $\alpha_V$ is a non-random,
finite Borel measure on $[0,\widetilde{T}]$ such that
$\alpha_V(\{0\})=0$. This is done by applying a change of
integration order similar to the one discussed in
\eqref{change_order}.
\end{enumerate}
\end{rem}

We conclude this section by observing that the flexibility in the
choice of the operators $\LLL^i$ opens the door to probabilistic
interpretations for many classes of integro-partial differential
equations, similarly to the well known non-linear Feynman-Kac
formula for BSDEs. We do not attempt to investigate this problem
into more detail here, leaving it for future research.

\section{Solvability on arbitrary intervals}\label{section_global_sol}

After proving in \autoref{exun_fbsde} the local solvability of the
fully coupled system \eqref{fbfunctdiffeq}, the next step consists
in showing the existence and uniqueness of solutions on arbitrarily
large time intervals. In the case of simple, non-coupled Lipschitz
functional differential equations, Liang et al. \cite{LiaLyoQia}
showed that the existence and uniqueness of solutions can be
obtained by imposing additional conditions on the operator $\LLL$,
while leaving the assumptions on the driver $f$ and the terminal
condition $\xi$ unchanged.

However, the situation is quite different for coupled systems: as it
is now well known in the theory of classical FBSDEs, an extension of
\autoref{exun_fbsde} to arbitrary intervals is possible only if we
impose additional assumptions on the coefficients $\mu$, $\sigma$,
$f$, $\phi$ of the fully coupled system \eqref{fbfunctdiffeq}, since
the solution could explode for large time horizons (without going
into more detail, we refer the reader to classical counterexamples
for Markovian FBSDEs and related PDEs which can be found for
instance in \cite{MaYon}). Several techniques have been adopted for
classical FBSDEs to overcome this difficulty, as we mentioned in the
Introduction. However, we prefer to apply another approach which in
our opinion is the most natural in the case of functional
differential systems.

First of all, we briefly discuss the intuition. Similarly to
\cite{LiaLyoQia}, the first step consists in dividing the interval
$[0,T]$ into a finite number of subintervals $I_j:=[T_{j-1}, T_j]$,
$0=T_0 < \dots < T_N=T$: then, we solve the system separately on any
subinterval, starting from the last one and going backward. There
is, however, an important additional difficulty with respect to
\cite{LiaLyoQia}, which clarifies why additional assumptions on the
coefficients are needed. For such an approach to work, we need in
fact that the length of the subintervals $I_j$, on which the system
has to be solvable, can be bounded by below by a constant
independent of $j$. We have seen that such a length depends on the
Lipschitz constants $C$ and $K$ of the system, and we can observe
that the only potential complication can arise from the terminal
condition on each $I_j$: indeed, while the other coefficients $\mu$,
$\sigma$, $f$ and $\LLL^i$ remain the same on all intervals $I_j$,
the terminal condition of the subsystem on $I_j$ is given by
$\Xi^j=\YYY(X^{j+1}, V^{j+1})_{T_j}$, where $(X^{j+1}, V^{j+1})$ is
the local solution on $I_{j+1}$. To obtain the desired lower bound
for the length for all $I_j$, it is therefore sufficient that
$\Xi^j=\theta_j(X^{j+1}_{T_j})$, where $\theta_j(\omega, x)$ is for
each $j$ Lipschitz continuous in $x$ with respect to some constant
$C$ \emph{independent of $j$}.

Due to the strong coupling between the two functional differential
equations, this uniform Lipschitz continuity can be obtained only by
studying the interplay between the two components $X$ and $V$.
However, this is very difficult without a concrete expression for
the operators $\LLL^i$, and is especially true in the case of
non-local operators (namely, the fact that $\LLL(M)$ could still
depend on the whole path of $(M_t)_{t \in [0,T]}$ may easily cause
the explosion of the solution for large time intervals). These
observations lead us to the conclusion that the solvability on
arbitrary intervals of the fully coupled system
\eqref{fbfunctdiffeq} has to be treated on a case-by-case basis, by
developing tailor-made techniques for each choice of the operators
$\LLL^i$.

As an example, we study in the following the case where the system
\eqref{fbfunctdiffeq} is associated to a classical Brownian FBSDE.
Hence, we will assume for the rest of the article that the
filtration is generated by an $m$-dimensional Brownian motion
$(W_t)_{t \in [0,T]}$ on $(\Om, \FFF,P)$, and that
$\LLL^1=\LLL^2=\LLL^3=\LLL$, where $\LLL$ is the operator given by
It\^o's representation theorem. For notational simplicity, we will
write $\ZZZ(X,V)=\LLL(\MMM(X,V))$.

In the literature of classical FBSDEs, it is critical to distinguish
between the cases where the coefficients $\mu$, $\sigma$, $f$ and
$\phi$ are purely deterministic (i.e. they do not depend on
$\omega$) or not. This is due to the fact that, in the purely
deterministic case, one can exploit the well known connection
between FBSDEs and parabolic PDEs. In the framework of functional
differential equations, such an approach has been adopted by Liang
et al. \cite{LiaLyoQia2} in a very special case, where they
essentially rely on the Lipschitz continuity of the solution of the
corresponding parabolic PDEs (see for example Delarue \cite{Del,
Del2}).


We are, however, more interested in the case where the coefficients
$\mu$, $\sigma$, $f$ and $\phi$ are allowed to be random. In the
literature, there are several articles dedicated to the random case:
in particular, Zhang \cite{Zha1,Zha2} proved the solvability of
classical FBSDEs by deriving uniform Lipschitz estimates with
respect to the initial condition of the component $X$. In the
following, after reformulating in our setting Zhang's results on
uniform Lipschitz continuity, we briefly illustrate how we can use
these results together with \autoref{exun_fbsde} to construct a
solution to our coupled functional differential system. \emph{For
the rest of this section, we will assume that $n=1$, i.e. the
component $X$ is $1$-dimensional, and that $\sigma$ does not depend
on $z_2$, i.e. $\sigma=\sigma(t,x,y)$.} Moreover, we assume that the
coefficients of \eqref{fbfunctdiffeq} satisfy the following
condition:

\bigskip
\noindent{\bfseries Assumption (A2):}  \emph{The functions $\mu:
\Omega \times[0,T]\times\RR\times \RR^d \times \RR^{d \times
m}\rightarrow \RR$, $\sigma: \Omega \times[0,T]\times \RR \times
\RR^d\rightarrow \RR^m$, $f: \Omega \times[0,T]\times \RR\times
\RR^d \times \RR^{d \times m} \rightarrow \RR^d$ and $\phi: \Omega
\times\RR\rightarrow \RR^d$ satisfy Assumption \textup{{\bfseries
(A2)}} if they satisfy Assumption \textup{{\bfseries (A1)}} with
\textup{(A1.2)} replaced by \textup{(A1.2')}, and:
\begin{itemize}
\item[\textup{(A2.1)}] \parbox[t][][l]{\mylength}{$\phi$ is uniformly Lipschitz in $x$ with constant $C'$.
}
\item[\textup{(A2.2)}] \parbox[t][][l]{\mylength}{There is a constant $\gamma > 0$ such that
\begin{align*}
\Lambda^1_t (y) &\le - \gamma |\Lambda^2_t (y)| \quad \forall \, y \in \{y' \in \RR^d \, | \, |y'|=1\}, \text{ where} \\
\Lambda^1_t (y)&:=\sum_{i=1}^d y_i \Big(\!\trace\big(\partial_z f^i (\partial_z \mu)^{\mathrm{T}} \big) - y^{\mathrm{T}}\partial_z \mu(\partial_z f^i)^{\mathrm{T}} y + y^{\mathrm{T}}\partial_y \sigma(\partial_z f^i)^{\mathrm{T}} y \Big) \\
&\phantom{:=}+ \partial_x \sigma(\partial_z \mu)^{\mathrm{T}} y + (\partial_y \mu)^{\mathrm{T}} y, \\
\Lambda^2_t (y)&:= |\partial_z \mu|^2 - |(\partial_z
\mu)^{\mathrm{T}} y|^2 + 2 y^{\mathrm{T}} \partial_z \mu (\partial_y
\sigma)^{\mathrm{T}} y,
\end{align*}
and where we assumed that all corresponding derivatives exist. }
\end{itemize}
}

The purpose of Assumption \textup{(A2.1)} is to make a distinction
between the Lipschitz constant of $\phi$ and those of the other
coefficients, since such a distinction is needed hereinafter.
Observe moreover that we require stronger regularity conditions than
those of Delarue \cite{Del,Del2}, since the coefficients of the
system \eqref{fbfunctdiffeq} are random (however, $\sigma$ is
allowed to be degenerate). Under these conditions, by deriving some
clever estimates for linear FBSDEs, Zhang obtained the following
result:
\begin{lem}[Zhang \cite{Zha2}]\label{zhang_lipschitz}
Assume that the coefficients $\mu$, $\sigma$, $f$ and $\phi$ satisfy
Assumption \textup{{\bfseries (A2)}} with $\gamma=\frac{1}{C}$, and
let $\ell$ denote the constant in \autoref{exun_fbsde}. Let $T <
\ell$, and for $x^i \in \RR$, $i=1,2$, let $(X^i, V^i)$ denote the
solution of
\[
\begin{cases}
d X^i_t= \mu(t,X^i_t, \YYY(X^i,V^i)_t, \ZZZ(X^i,V^i)_t) d t + \sigma(t,X^i_t,\YYY(X^i,V^i)_t) d W_t,  \\
d V^i_t=f(t,X^i_t, \YYY(X^i,V^i)_t, \ZZZ(X^i,V^i)_t) d t, \\
X^i_0=x^i, \quad V^i_0=0.
\end{cases}
\]
Then, there is a constant $\varrho_C$, depending only on $C$, such
that
\[
\big|\YYY(X^1,V^1)_0 - \YYY(X^2,V^2)_0\big| \le \overline{C} |x^1-
x^2|,
\]
where $\overline{C}:=\sqrt{\big(|C'|^2 + 1\big) e^{\varrho_C T} -1}
> 0$.
\end{lem}

The proof can be found in \cite{Zha2}. As already mentioned,
\autoref{zhang_lipschitz} in conjunction with \autoref{exun_fbsde}
allows to construct a solution to the system \eqref{fbfunctdiffeq}
on arbitrarily large time intervals. More exactly:
\begin{theorem}\label{global_fbsdyn}
Assume that the coefficients $\mu$, $\sigma$, $f$ and $\phi$ satisfy
Assumption \textup{{\bfseries (A2)}}. Then, for any $T > 0$, the
fully coupled system \eqref{fbfunctdiffeq} has a unique solution
$(X,V)$ in $\SSSS^2([0,T],\RR^n) \times \SSSS^2([0,T],\RR^d)$.
\end{theorem}

\begin{proof}
We assume without loss of generality that $\gamma=\frac{1}{C}$ in
Assumption {\bfseries (A2)}, by changing $C$ or $\gamma$ if
necessary.  Let $\overline{C}$ denote the constant in
\autoref{zhang_lipschitz}, and let $\ell=\ell(\overline{C})$ be the
constant in \autoref{exun_fbsde}. We consider a partition $(T_0,
\cdots, T_N)$ of $[0,T]$ such that $0=T_0 < \cdots < T_N=T$ and  $0
< T_i - T_{i-1} < \ell$ for $i=1,\cdots,N$, and set $I_i:=[T_{i-1},
T_i]$. Moreover, to emphasize the dependence of the operators
$\YYY$, $\MMM$ on the terminal time and condition, we introduce a
slightly different notation and write, for all $t \in [0,T_i]$,
\begin{align*}
\breve{\MMM}^{T_i}(\xi,V)_t&=E[\xi + V_{T_i}|\FFF_t], \\
\breve{\YYY}^{T_i}(\xi,V)_t&=E[\xi + V_{T_i}|\FFF_t] - V_t.
\end{align*}
$\breve{\ZZZ}^{T_i}(\xi,V)_t$ is then defined via the It\^o
representation of $(\breve{\MMM}^{T_i}(\xi,V)_t)_{t \in [0,T_i]}$.

The first step consists in constructing appropriate terminal
conditions for all subintervals $I_i$. This is accomplished via a
backward procedure. We set $\theta_N:=\phi$, $C_N:=C'$, and for all
$x \in \RR$, we consider the following system on $I_N$: \small
\[
\begin{cases}
d X^{N,x}_t= \mu(t,X^{N,x}_t, \breve{\YYY}^{T_N}(\theta_N(X^{N,x}_{T_N}),V^{N,x})_t, \breve{\ZZZ}^{T_N}(\theta_N(X^{N,x}_{T_N}),V^{N,x})_t) d t \\
\phantom{d X^{N,x}_t=} + \sigma(t,X^{N,x}_t,\breve{\YYY}^{T_N}(\theta_N(X^{N,x}_{T_N}),V^{N,x})_t) d W_t,  \\
d V^{N,x}_t=f(t,X^{N,x}_t, \breve{\YYY}^{T_N}(\theta_N(X^{N,x}_{T_N}),V^{N,x})_t, \breve{\ZZZ}^{T_N}(\theta_N(X^{N,x}_{T_N}),V^{N,x})_t) d t, \\
X^{N,x}_{T_{N-1}}=x, \quad V^{N,x}_{T_{N-1}}=0.
\end{cases}
\]
\normalsize Since $\theta_N$ has Lipschitz constant $C_N \le
\overline{C}$, the system has for all $x \in \RR$ a unique solution
$(X^{N,x},V^{N,x})$ by \autoref{exun_fbsde} and
\autoref{rem_initialtime} \textup{(}\textit{i}\textup{)}. We can
thus define $\theta_{N-1}$ by
\[
\theta_{N-1}(x):=\breve{\YYY}^{T_N}(\theta_N(X^{N,x}_{T_N}),
V^{N,x})_{T_{N-1}}.
\]
It is then easy to check that $\theta_{N-1}(x)$ is
$\FFF_{T_{N-1}}$-measurable for all $x \in \RR$. Moreover, by
\autoref{zhang_lipschitz}, $\theta_{N-1}$ is uniformly Lipschitz in
$x$ with constant
\[
C_{N-1}:=\sqrt{\big(|C_N|^2 + 1\big) e^{\varrho_C (T_N - T_{N-1})}
-1}.
\]

Since $C_{N-1} \le \overline{C}$, we can iterate the same argument:
for $i=N-1,\cdots, 2$, we consider for all $x \in \RR$ the solution
$(X^{i,x},V^{i,x})$ on $I_i$ of the system: \small
\[
\begin{cases}
d X^{i,x}_t= \mu(t,X^{i,x}_t, \breve{\YYY}^{T_i}(\theta_i(X^{i,x}_{T_i}),V^{i,x})_t, \breve{\ZZZ}^{T_i}(\theta_i(X^{i,x}_{T_i}),V^{i,x})_t) d t \\
\phantom{d X^{i,x}_t=} + \sigma(t,X^{i,x}_t,\breve{\YYY}^{T_i}(\theta_i(X^{i,x}_{T_i}),V^{i,x})_t) d W_t,  \\
d V^{i,x}_t=f(t,X^{i,x}_t, \breve{\YYY}^{T_i}(\theta_i(X^{i,x}_{T_i}),V^{i,x})_t, \breve{\ZZZ}^{T_i}(\theta_i(X^{i,x}_{T_i}),V^{i,x})_t) d t, \\
X^{i,x}_{T_{i-1}}=x, \quad V^{i,x}_{T_{i-1}}=0.
\end{cases}
\]
\normalsize which exists by \autoref{exun_fbsde} and
\autoref{rem_initialtime} \textup{(}\textit{i}\textup{)}. We then
define $\theta_{i-1}$ by
\[
\theta_{i-1}(x):=\breve{\YYY}^{T_i}(\theta_i(X^{i,x}_{T_i}),
V^{i,x})_{T_{i-1}}.
\]
$\theta_{i-1}(x)$ is thus $\FFF_{T_{i-1}}$-measurable for all $x \in
\RR$, and by \autoref{zhang_lipschitz}, $\theta_{i-1}$ is uniformly
Lipschitz in $x$ with constant $C_{i-1}:=\sqrt{\big(|C_i|^2 + 1\big)
e^{\varrho_C (T_i - T_{i-1})} -1}$. Moreover, we can easily verify
by induction that
\[
C_{i-1}= \sqrt{\big(|C_N|^2 + 1\big) e^{\varrho_C (T_N - T_{i-1})}
-1} \le \overline{C}.
\]

Now that we have derived appropriate terminal conditions $\theta_i$,
we can construct the solution on the whole interval $[0,T]$ by a
forward procedure. We set $X^0_{T_0}:=x$, $V^0_{T_0}:=0$. Then, for
$i=1, \cdots, N$, we denote by $(X^i,V^i)$ the solution on $I_i$ of
the system \small
\[
\begin{cases}
d X^i_t= \mu(t,X^i_t, \breve{\YYY}^{T_i}(\theta_i(X^i_{T_i}),V^i)_t, \breve{\ZZZ}^{T_i}(\theta_i(X^i_{T_i}),V^i)_t) d t \\
\phantom{d X^i_t=} + \sigma(t,X^i_t,\breve{\YYY}^{T_i}(\theta_i(X^i_{T_i}),V^i)_t) d W_t,  \\
d V^i_t=f(t,X^i_t, \breve{\YYY}^{T_i}(\theta_i(X^i_{T_i}),V^i)_t, \breve{\ZZZ}^{T_i}(\theta_i(X^i_{T_i}),V^i)_t) d t, \\
X^i_{T_{i-1}}=X^{i-1}_{T_{i-1}}, \quad
V^i_{T_{i-1}}=V^{i-1}_{T_{i-1}}.
\end{cases}
\]
\normalsize which exists due to \autoref{exun_fbsde} and
\autoref{rem_initialtime} \textup{(}\textit{i}\textup{)}. We then
set, for $t \in I_i$,
\[
X_t:=X^i_t, \quad V_t:=V^i_t.
\]
To prove that $(X,V)$ solves \eqref{fbfunctdiffeq} on $[0,T]$, it
suffices to check that, for $t \in I_i$,
\[
\breve{\YYY}^{T_i}(\theta_i(X^i_{T_i}),V^i)_t=\YYY(\phi(X_{T_N}),V)_t,
\quad
\breve{\ZZZ}^{T_i}(\theta_i(X^i_{T_i}),V^i)_t=\ZZZ(\phi(X_{T_N}),V)_t.
\]
However, for $i=1,\cdots, N-1$ and $t\in [0,T_i]$, we have that
\begin{align*}
\breve{\YYY}^{T_i}(\theta_i(X_{T_i}),V)_t
&=\breve{\YYY}^{T_i}(\theta_i(X^i_{T_i}),V^i)_t\\
&=E\big[\YYY^{T_{i+1}}\big( \theta_{i+1}(X^{i+1}_{T_{i+1}}), V^{i+1}\big)_{T_i} + V^i_{T_i} \big|\FFF_t\big] - V^i_t \\
&= E\big[\theta_{i+1}(X_{T_{i+1}})+ V_{T_{i+1}} - V_{T_i} + V_{T_i} \big|\FFF_t\big] - V_t \\
&= \breve{\YYY}^{T_{i+1}}(\theta_{i+1}(X_{T_{i+1}}),V)_t
\end{align*}
by the construction of $\theta_i$. This gives by induction that
$\breve{\YYY}^{T_i}(\theta_i(X^i_{T_i}),V^i)_t=\YYY(\theta_N(X_{T_N}),V)_t=\YYY(\phi(X_{T_N}),V)_t$
on $I_i$ for $i=1,\cdots,N$. On the other hand, this implies that
$\breve{\MMM}^{T_i}(\theta_i(X^i_{T_i}),V^i)_t=\MMM(\phi(X_{T_N}),V)_t$
on $I_i$ for all $i$. In other words,
$(\breve{\MMM}^{T_i}(\theta_i(X^i_{T_i}),V^i)_t)_{t \in I_i}$ is the
restriction to $I_i$ of the martingale $(\MMM(\phi(X_{T_N}),V)_t)_{t
\in [0,T]}$ and, due to the locality of the operator $\ZZZ$, this
gives us that
$\breve{\ZZZ}^{T_i}(\theta_i(X^i_{T_i}),V^i)_t=\ZZZ(\phi(X_{T_N}),V)_t$
on $I_i$.

This shows that $(X,V)$ is a solution of \eqref{fbfunctdiffeq} on
$[0,T]$, and the proof is concluded by observing that the uniqueness
is a consequence of the uniqueness of $(X^i,V^i)$ on $I_i$.
\end{proof}

We conclude this article by briefly mentioning an extension of
Zhang's results \cite{Zha1,Zha2} recently derived by Ma et al.
\cite{MaWuZhaZha}. Motivated by the connection between FBSDEs and
PDEs in the deterministic case, the authors suggest that, for random
coefficients, the solution can be extended to arbitrary intervals by
relying on the existence of a random field $\theta$ (called
decoupling field) such that $\YYY(X,V)_t=\theta(t, X_t)$. It is then
sufficient to show that $\theta$ is uniformly Lipschitz continuous.
This can be obtained via the introduction of a backward stochastic
Riccati equation of quadratic growth, under much weaker assumptions
than Assumption {\bfseries (A2)} (however, all processes have to be
one-dimensional).


\end{document}